\documentclass{article}
\usepackage[utf8]{inputenc}
\usepackage[T1]{fontenc}
\usepackage{amsmath,amssymb,amsthm,bbm,float,graphicx,geometry,lmodern,mathtools,parskip,setspace,subcaption}
\usepackage[colorlinks, citecolor = red!80!black, linkcolor = blue!80!black, breaklinks]{hyperref}
\usepackage{mathrsfs}
\usepackage{enumerate}
\usepackage{nicefrac}
\usepackage{todonotes}
\usepackage{comment}
\usepackage{subcaption}
\usepackage{titling}
\usepackage{fullpage}
\usepackage{orcidlink}

\usepackage{caption}
\usepackage{graphicx}
\graphicspath{{Images/}}
\allowdisplaybreaks

\newcommand{\N}{\mathbb{N}}

\newtheorem{theorem}{Theorem}[section]
\newtheorem{lemma}[theorem]{Lemma}

\theoremstyle{definition}

\newtheorem{remark}[theorem]{Remark}

\usepackage{tikz}
\usetikzlibrary{backgrounds}
\usetikzlibrary{patterns}
\usetikzlibrary{positioning, shapes.geometric}
\usetikzlibrary{external, arrows, decorations.pathmorphing} 

\renewcommand{\P}{\mathbb{P}}
\newcommand{\x}{\mathbf{x}}
\newcommand{\y}{\mathbf{y}}
\newcommand{\1}{\mathbbm{1}}
\newcommand{\G}{\mathscr{G}}
\newcommand{\C}{\mathscr{C}}
\newcommand{\Mcal}{\mathscr{M}}
\newcommand{\Ncal}{\mathscr{N}}

\newcommand{\E}{\mathbb{E}}
\newcommand{\0}{\mathbf{o}}
\newcommand{\R}{\mathbb{R}}
\renewcommand{\d}{\mathrm{d}}
\newcommand{\z}{\mathbf{z}}

\newcommand{\Z}{\mathbb{Z}}

\newcommand{\cX}{\mathcal{X}}

\newcommand{\deff}{\delta_{\mathrm{eff}}}
\newcommand{\B}{\mathcal{B}}

\pretitle{\centering\LARGE\scshape}
 \posttitle{\vskip 0.75cm}

 \predate{\vskip 0.75 cm \centering\large}
 \postdate{\par}

\usepackage[style = numeric, sorting=nyt, url = false, abbreviate=false, maxbibnames=9, sortcites=true, doi = true, backend = biber, giveninits = true]{biblatex}
\renewbibmacro{in:}{\ifentrytype{article}{}{\printtext{\bibstring{in}\intitlepunct}}}
\bibliography{bib.bib}

\thanksmarkseries{arabic}

\author{
Benedikt Jahnel \orcidlink{0000-0002-4212-0065}\thanks{Technische Universit\"at Braunschweig, Universit\"atsplatz 2, 38106 Braunschweig, Germany}\thanksgap{0.4ex} \thanks{Weierstrass Institute for Applied Analysis and Stochastics, Mohrenstr.\ 39, 10117 Berlin, Germany} \\ benedikt.jahnel@tu-braunschweig.de
\and
Lukas L\"{u}chtrath \orcidlink{0000-0003-4969-806X}\thanksmark{2}\\lukas.luechtrath@wias-berlin.de
}

\title{Existence of subcritical percolation phases for generalised weight-dependent random connection models}

\date{September 7, 2023}

\usepackage{etoolbox}
\ifundef{\abstract}{}{\patchcmd{\abstract}%
{\quotation}{\quotation\noindent\ignorespaces}{}{}}
\begin{document}
\maketitle 
\begin{spacing}{0.9}
\begin{abstract} We derive a sufficient condition for the existence of a subcritical percolation phase for a wide range of continuum percolation models where each vertex is embedded into Euclidean space according to an iid-marked stationary Poisson point process. In contrast to many established models, the probability of existence of an edge may not only depend on the distance and the weights of its end vertices but also on a surrounding vertex set. Our results can be applied in particular to models combining heavy-tailed degree distributions and long-range effects, which are typically well connected. 

\noindent More precisely, we study the critical annulus-crossing intensity \(\widehat{\lambda}_c\) which is smaller or equal to the classical critical percolation intensity \(\lambda_c\) and derive sharp conditions for \(\widehat{\lambda}_c>0\) by controlling the occurrence of long edges. We further present tail bounds for the Euclidean diameter and number of points of the typical cluster in the subcritical phase and apply our results to several examples including some in which \(\widehat{\lambda}_c<\lambda_c\).

\smallskip
\noindent\footnotesize{{\textbf{AMS-MSC 2020}: 60K35}

}
\end{abstract}
\end{spacing}

\section{Introduction}
The standard objects studied in continuum percolation theory are random graphs \(\G^\lambda\) on the points of a homogeneous Poisson point process on \(\R^d\) of intensity \(\lambda>0\). The spatial embedding of the vertices enters the connection probability in a way that spatially close vertices are connected by an edge with a higher probability than far apart vertices. Many well established models belong to that framework, i.e., the Gilbert graph~\cite{Gilbert61}, the random connection model~\cite{MeesterPenroseSarkar1997,Penrose2016}, the Poisson--Boolean model~\cite{Hall85,Gouere08} and its soft version~\cite{GGM2022}, continuum scale-free percolation~\cite{DeijfenHofstadHooghiemstra2013,DeprezWuthrich2019}, or the age-dependent random connection model~\cite{GraLuMo2022}. The standard question in percolation theory is then whether there exists a critical Poisson intensity \(\lambda_c\in(0,\infty)\) such that the connected component of the origin (added to the graph if necessary) is infinite with a positive probability for all \(\lambda>\lambda_c\) but is finite almost surely for \(\lambda<\lambda_c\). We call the regime \((0,\lambda_c)\) the \emph{subcritical percolation phase} and \((\lambda_c,\infty)\) the \emph{supercritical percolation phase}. Often, by ergodicity of the underlying Poisson point process and the way edges are drawn, if \(\lambda>\lambda_c\) then \(\G^\lambda\) contains an infinite connected component almost surely and if \(\lambda<\lambda_c\) there cannot be an infinite connected component somewhere in the graph. Moreover, under very mild assumptions on the distribution of \(\G^\lambda\), an existing infinite component is almost surely unique~\cite{BurtonKeane89}. 

In dimensions \(d\geq 2\), percolation models typically contain a supercritical phase \cite{MeesterRoy1996} which is essentially a consequence of the existence of a supercritical percolation phase in nearest-neighbour Bernoulli percolation on \(\Z^2\)~\cite{Grimmett1999}. Therefore, for \(d\geq 2\), the proof of existence of a non-trivial phase-transition \(\lambda_c\in(0,\infty)\) reduces to proving the existence of a subcritical phase. In this article, we present sufficient conditions for the existence of a subcritical percolation phase in a quite general setting. Under these conditions, we are able to exhibit estimates for the tail behavior of the distribution of the Euclidean diameter and the number of points of the component of a typical point in the subcritical regime. 

To prove the existence of a subcritical phase in the Boolean model, where each vertex is assigned an i.i.d.\ radius and any pair of vertices is connected whenever their associated balls intersect, Gou\'{e}r\'{e} introduced an alternative critical intensity \(\widehat{\lambda}_c\)~\cite{Gouere08}. It is defined as the infimum of all intensities for which the probability of existence of a path that connects a vertex in an \(\alpha^{1/d}\) neighbourhood of the origin to a vertex at distance at least \(2\alpha^{1/d}\) is uniformly bounded from zero. This then gives rise to multi-scale arguments which can then also be used to derive tail bounds for the Euclidean diameter or the number of points in the new subcritical phase. In recent works the equality of \(\widehat\lambda_c\) and \(\lambda_c\) was studied~\cite{AhlbergTassionTeixeira2018,DCopinRaoufiTassion2020,DembinTassion2022} and it is shown that indeed \(\widehat\lambda_c=\lambda_c\) for all but at most countable many radius distributions. 

To prove the existence of supercritical percolation phases in one-dimensional models, Gracar et al.~\cite{GraLuMo2022} recently introduced a new coefficient \(\deff\). This new coefficient measures the occurrence of long edges in a way comparable to classical long-range percolation models, seen from a coarse-grained perspective. It was shown that \(\deff\) can further be used to derive transience of a supercritical random walk and continuity of the percolation probability~\cite{Moench2023}, indicating the importance of this quantity to analyse these kind of inhomogeneous random graphs. We show in this paper that \(\deff\) can also be used to determine whether \(\widehat\lambda_c>0\) or not and we shall see that unlike in the Boolean model the equality of \(\lambda_c\) and \(\widehat\lambda_c\) is not generally true. More precisely, we show that for \(\deff>2\) the arguments of~\cite{Gouere08} can be adapted yielding \(\widehat{\lambda}_c>0\). However, for \(\deff<2\) we prove \(\widehat{\lambda}_c=0\) and present examples for which \(\deff<2\) but \(\lambda_c>0\) at the same time.

 In~\cite{GraLuMo2022}, the coefficient~\(\deff\) is derived for the \emph{weight-dependent random connection model}~\cite{GHMM2022}. In this class of models, containing all the aforementioned models, each vertex carries an independent mark distributed uniformly on \((0,1)\). The connection mechanism is such that edges are drawn independently given the vertex locations and their marks. Additionally, connections to spatially close vertices or vertices with small marks are preferred. Here, the first preference leads to clustering and the latter can be used to entail heavy-tailed degree distributions to the graph. This is done in a way that the degree-distribution is modelled independently of the clustering, i.e., the degree distribution depends only on the way the vertex marks enter the connection probability but does not depend on the geometric restrictions. We build on their work but extend the setting to models where both effects are allowed to depend on each other. Moreover, we additionally allow the edges to depend on local neighbourhoods of their end vertices. This particularly include models without the typical monotony assumption on the intensity. That is, additional vertices may not increase the percolation probability. Our result then shows that for small enough intensities no infinite component exists. However, it may be the case that the same be true for too large intensities. Put differently, without the standard monotonicity assumption, there may be no longer a unique subcritical and supercritical percolation phase. We present an example for a non-monotone model below.

\subsection{Framework}
We consider graphs where the vertex set is given by a standard Poisson point process on \(\R^d\) of intensity \(\lambda>0\). Each vertex carries an independent mark distributed uniformly on \((0,1)\). We denote a vertex by \(\mathbf{x}=(x,u_x)\) and refer to \(x\in\R^d\) as the vertex's \emph{location} and to \(u_x\in(0,1)\) as the vertex's \emph{mark}. 
We denote the set of all marked vertices by \(\cX=\cX^\lambda\) and remark that \(\cX\) is a standard Poisson point process on \(\R^d\times(0,1)\) of intensity \(\lambda>0\) \cite{LastPenrose2017}. 
Given \(\cX\), a pair of vertices \(\x=(x,u_x)\) and \(\y=(y,u_y)\in\cX\) is connected by an edge with probability \(\mathbf{p}(\x,\y,\cX\setminus\{\x,\y\})=\mathbf{p}(\y,\x,\cX\setminus\{\x,\y\})\). That is, whether an edge is drawn depends not only on the potential end vertices of the edge but may also depend on all other vertices. We denote the underlying probability measure containing the marked Poisson point process and the randomness of the edges by \(\P^\lambda\) and its expectation by \(\E^\lambda\). We assume that \(\mathbf{p}\) is invariant under translations and rotations, and that it fulfills the following homogeneity condition when integrated with respect to the underlying Poisson point process: For two deterministically given vertices \(\x\) and \(\y\), we have  
\begin{equation}
    \E^\lambda[\mathbf{p}(\x,\y,\cX)]= \varphi(u_x, u_y, |x-y|^d) \label{eq:varphi},
\end{equation}
where \(\varphi:(0,1)\times(0,1)\times(0,\infty)\to[0,1]\) is a measurable function. Note that the deterministically given vertices are no elements of the Poisson point process almost surely. We make the following assumptions on the function \(\varphi\): 
\begin{enumerate}[(i)]
    \item \(\varphi\) is symmetric in the first two arguments and non-increasing in the third argument. Hence, connections to spatially close vertices are preferred. 
    \item The integral
        \[\int_0^1\int_0^1\int_0^\infty \varphi (s,t,r) \d r\, \d s \, \d t\]
    is finite. This ensures that expected degrees are finite. 
\end{enumerate}
 Condition~\eqref{eq:varphi} essentially says that the annealed probability of two given vertices being connected only depends on the given vertices' marks and their distance. Indeed, since \(\cX\) is a homogeneous process and \(\mathbf{p}\) is translation, and rotation invariant, the Poisson point process and its influence on the connection probability looks in expectation everywhere the same. 

We denote the resulting undirected graph by \(\G=\G^\lambda\). We denote the event that \(\x\) and \(\y\) are connected by an edge by \(\x\sim\y\) and that they belong to the same connected component by \(\x\leftrightarrow\y\). For a given vertex \(\x\), we denote by \(\mathscr{C}(\x)\) the connected component of \(\x\). Moreover, for two positive functions \(f\) and \(g\) we write \(f\asymp g\) to indicate that \(f/g\) is bounded from zero and infinity. We denote by \(\sharp A\) the number of elements in a countable set \(A\). We further use constants \(c,C>0\) throughout the manuscript. These constants mainly occur as integration constants and may change from line to line.

\subsection{Main result}
To formulate our main result, we work with the Palm version \cite{LastPenrose2017} of the model. That is, a distinguishable typical vertex \(\mathbf{o}=(o,U_o)\) is placed at the origin \(o\), marked with an independent uniform random variable \(U_o\) and is then added to the vertex set. The graph \(\G_o=\G_o^\lambda\) is then constructed as above but now with the additional vertex \(\mathbf{o}\) and we denote the underlying law by \(\P_o^\lambda\). We define the random variables
\begin{equation*}
    \begin{aligned}
        & \C:=\C(\mathbf{o}) = \{\x\in\cX: \mathbf{o}\leftrightarrow\x\}, \\
        & \Mcal:= \Mcal(\mathbf{o})= \sup\{|x|^d: \x\in\C(\mathbf{o})\}, \ \text{ and } \\
       	 & \Ncal:= \Ncal(\mathbf{o})= \sharp\C(\mathbf{o}).
    \end{aligned}
\end{equation*}
We further define as
\begin{equation*}
	\begin{aligned}
		\lambda_c:=\lambda_c(\mathbf{p}) & = \inf \big\{\lambda>0:\P_o^\lambda(\sharp\C(\0)=\infty)>0\big\}= \inf \big\{\lambda>0:\lim_{\alpha\to\infty}\P_o^\lambda(\exists \x:|x|^d>\alpha, \text{ and }\0\leftrightarrow\x)>0\big\}
	\end{aligned}
\end{equation*}
the {\em classical critical percolation intensity} and as
\begin{equation*}
	\begin{aligned}
		\widehat{\lambda}_c:=\widehat{\lambda}_c(\mathbf{p})=\inf\big\{\lambda>0: \lim_{\alpha\to\infty}\P^\lambda(\exists \x,\y: |x|^d<\alpha, |y|^d>2^d \alpha, \text{ and }\x\leftrightarrow\y)>0\big\}
	\end{aligned} 
\end{equation*}
the {\em critical intensity} introduced in~\cite{Gouere08} as discussed in the introduction. It is clear that \(\widehat{\lambda}_c\leq\lambda_c\).
 
To ensure the positivity of \(\widehat{\lambda}_c\), we rely on two features our graph has to provide: First, the number of `long edges' should be sufficiently small, which then yields that percolation must happen locally in some sense. Second, the influence of the whole vertex set on the connection mechanism should only be driven by spatially close vertices to ensure that local percolation in two balls at a large distance can be seen as independent.

Following~\cite{GraLuMo2022}, the occurrence of `long edges' is measured by the following quantity, called \emph{the effective decay exponent} (associated with \(\varphi\))
\begin{equation}\label{eq:DefDeff}
    \begin{aligned}
        	\deff:=\deff(\varphi)=-\lim_{n\to\infty}\frac{\log \int_{\nicefrac{1}{n}}^1\int_{\nicefrac{1}{n}}^1 \varphi(s,t,n)\d s \d t}{\log n}.
    \end{aligned}
\end{equation}
Imagine two sets consisting of \(n\) vertices each at distance roughly \(n^{1/d}\). If \(n\) is large, the smallest mark one would expect to find in each set is roughly of size \(1/n\). Hence, the probability that two vertices, one of which chosen randomly from one of the sets and the other vertex independently picked from the other set, is roughly given by the numerator's integral where the influence of the surrounding vertices has already been integrated out. If the limit exists, this probability decays as \(n^{-\deff}\). Ignoring the correlations involved and treating all edges as being independent, the expected number of edges between the two sets is then given by \(n^{2-\deff}\) which grows large for \(\deff<2\) and decays to zero for \(\deff>2\). For a more detailed discussion on \(\deff\), we refer the reader to~\cite{GraLuMo2022}. Since we cannot a priori assume that this limit always exists and additionally we have to deal with random fluctuations in the vertex marks, we define for \(\mu>0\) the two following limits 
\begin{equation*}
	\begin{aligned}
		& \psi^-(\mu):=\psi^-(\mu,\psi)=-\limsup_{n\to\infty} \frac{\log \int\limits_{n^{-1-\mu}}^1\int\limits_{n^{-1-\mu}}^1 \varphi(s,t,n)\d s \d t}{\log n} \quad \text{ and } \\
		& \psi^+(\mu):=\psi^+(\mu,\psi)=-\liminf_{n\to\infty} \frac{\log \int\limits_{n^{-1+\mu}}^1\int\limits_{n^{-1+\mu}}^1 \varphi(s,t,n)\d s \d t}{\log n}.
	\end{aligned}
\end{equation*} 
These are then used to define
\[
	\deff^-:=\deff^-(\varphi)=\lim_{\mu\downarrow 0}\psi^-(\mu,\varphi) \quad \text{ and } \quad \deff^+:=\deff^+(\varphi)=\lim_{\mu\downarrow 0}\psi^+(\mu,\varphi).
\]
If both of the latter limits coincide, we observe the primarily introduced coefficient \(\deff\). The idea of the additional parameter \(\mu\) is to allow more flexibility for the smallest found vertex mark. In order to show that the existence of long edges is unlikely, i.e., \(\deff^->2\), we allow slightly smaller marks, whereas for the probability of existence of long edges, i.e., \(\deff^+>2\), we only consider slightly larger marks. To determine the influence of the mark's fluctuations on \(\deff\), we further define
\[
	\bar{\mu} := \sup\{\mu>0: \mu< \psi^-(\mu)-2\}.
\]
We do not define the analogue for \(\psi^+\) as we will not need this quantity throughout the manuscript. Let us comment here on our monotonicity assumptions on the connection function \(\varphi\). Intuitively, one may want to consider the marks as inverse weights and ask \(\varphi\) not only to be non-increasing in the `distance argument' but in the `mark arguments' as well. We do not need this strong monotonicity assumption for our proof to work but some kind of monotonicity is implicitly build in by our reliance on \(\deff\). Indeed, to have \(\deff>2\), we have the requirement that for connecting vertices at distance \(n^{1/d}\), the vertices with marks roughly \(\nicefrac{1}{n}\) have to be the dominant ones, as pointed out above. We do however allow connection mechanisms that `punish'{} vertices with atypically small marks. We present such a model below in Section~\ref{sec:example}.

For our second requirement on the graph, let us write \(\mathcal{B}(r,x)\) for the open ball of radius \(r\) centered in \(x\) and \(\mathcal{B}(r):=\mathcal{B}(r,o)\). To specify local percolation and quantify the influence of far-apart vertices on the connection mechanism, we need to introduce some notation. For measureable domains \(D\subset\R^d\) and \(I\subset(0,1)\) we write  
\[
    \cX(D\times I)=\{\x=(x,u_x)\in\cX\colon x\in D, u_x\in I\}.
\] 
If \(I=(0,1)\), we simply write \(\cX(D)=\cX(D\times(0,1))\). Further, we denote by \(\mathscr{C}_D(\x)\) the connected component of \(\x\) restricted to the subgraph on the vertices of \(\cX(D)\). For a given location \(x\in\R^d\) and \(\alpha>1\), we define the event
\begin{equation} \label{eq:GEvent}
    G(\alpha,x)= \big\{\exists \y\in\cX(\mathcal{B}({\alpha^{1/d}},x)):\C_{\mathcal{B}({3\alpha^{1/d}},x)}(\y)\not\subset\cX(\mathcal{B}({2\alpha^{1/d}},x))\big\}.
\end{equation}
That is, the vertex \(\y\) located close to \(x\) reaches with a path a vertex located in the annulus \(\mathcal{B}(3 \alpha^{1/d},x)\setminus\mathcal{B}(2\alpha^{1/d},x)\) without using vertices located outside \(\mathcal{B}(3\alpha^{1/d},x)\). We abbreviate \(G(\alpha) = G(\alpha,o)\).
We say that \(\G^\lambda\) is \emph{(polynomial) mixing}, with index \(\zeta>0\), if there exists a constant \(C_\text{mix}>0\) such that for all \(\lambda>0\) and all \(|x|>6\alpha^{1/d}\), we have
\begin{equation}\label{eq:mixing}
    \big|\operatorname{Cov}\big(\1_{G({\alpha})},\1_{G(\alpha,x)}\big)\big| \leq C_\text{mix} \, \lambda \alpha^{-\zeta}.
\end{equation}
Note that the examples mentioned in the introduction are all mixing in our sense as all of them have \(\mathbf{p}(\y,\x,\cX\setminus\{\x,\y\})=\varphi(u_x,u_y,|x-y|^d)\). We identify this case with \(\zeta=\infty\). 

\medskip

\begin{theorem}[The subcritical phase] \label{thm:Subcritical} Consider the graph \(\G^\lambda\) with connection probability \(\mathbf{p}\) such that \(\G^\lambda\) is mixing with index \(\zeta>0\), and \(\deff^-(\varphi)>2\).
   \begin{enumerate}[(i)]
   	\item  We have \(\widehat{\lambda}_c(\mathbf{p})>0\) and 

   	\item for all \(\lambda<\widehat{\lambda}_c(\mathbf{p})\) and \(a<\zeta\wedge\bar{\mu}\)
   	\[
   		\E^\lambda \Mcal^a <\infty \qquad \text{ and } \qquad \E^\lambda \Ncal^a <\infty. 
   	\]
   \end{enumerate}
\end{theorem}

\medskip

\begin{remark}~\
	\begin{enumerate}[(a)]
		\item Part~\textit{(ii)} of the Theorem shows that for \(\lambda<\hat{\lambda}\), the tails of \(\Mcal\) and \(\Ncal\) decay no faster than polynomial with index \(\zeta\wedge\bar{\mu}\) up to potentially slowly varying corrections. 
		\item In order to derive \(\widehat{\lambda}_c>0\) in Part \textit{(i)}, it suffices to replace the right-hand side in~\eqref{eq:mixing} with \(\lambda g(\alpha)\) where \(g(\alpha)\) tends to zero at an arbitrary speed. If this is the case, we derive in Part~\textit{(ii)} moment bounds depending on \(\bar{\mu}\) and \(g(\alpha)\). More precisely, if \(g(\alpha)\) decays slower than any polynomial, the of decay of the tail probabilities of \(\Mcal\) and \(\Ncal\) is bounded by \(g(\alpha)\) up to slowly varying corrections. 
	\end{enumerate}
	
\end{remark} 

Our second result shows that for \(\deff<2\) and \(\zeta=\infty\), we always have \(\widehat{\lambda}_c=0\). We present examples in Section~\ref{sec:example} with \(\lambda_c>0\) and \(\deff<2\) showing that \(\widehat{\lambda}_c\) cannot always be used to derive the existence of subcritical phases. 

\medskip 

\begin{theorem} \label{thm:NonSharpness} 
	Consider the graph \(\G^\lambda\) with connection probability \(\mathbf{p}\). Assume further \(\deff^+(\varphi)<2\) and \(\zeta=\infty\), i.e., edges are drawn independently from the surrounding vertices, then we have \(\widehat{\lambda}_c(\mathbf{p})=0\). 	
\end{theorem}

\medskip

We present the proofs of the Theorems~\ref{thm:Subcritical} and~\ref{thm:NonSharpness} in Section~\ref{sec:Proofs}. 

\subsection{Examples}\label{sec:example}
\paragraph{The weight-dependent random connection model}
This model was introduced in~\cite{GHMM2022} and further studied in~\cite{GLM2021, GraLuMo2022,GGM2022,Moench2023}. Here, given \(\cX\), edges are drawn independently from one another and we have
\begin{align}\label{ex_1}
        \mathbf{p}(\x,\y,\cX\setminus\{\x,\y\})=\varphi(u_x,u_y,|x-y|^d)=\rho\big(\beta^{-1} g(u_x,u_y)|x-y|^d\big),
\end{align}
where \(\rho:(0,\infty)\to[0,1]\) is an integrable and non-increasing \emph{profile function} and \(g:(0,1)^2\to(0,\infty)\) is a non-increasing \emph{kernel function} which is symmetric in both arguments. Further, \(\beta>0\) controls the edge intensity of the graph by scaling the vertices' distance in the connection probability. By the Poisson point process mapping theorem~\cite{LastPenrose2017}, it is no loss of generality to fix \(\beta=1\) and only vary the Poisson intensity \(\lambda\) or doing the opposite and fixing \(\lambda=1\) whilst varying \(\beta\). Two types of profile functions have been established in the literature: The \emph{long-range} profile function \(\rho(x):=p(1\wedge |x|^{-\delta})\) for \(\delta> 1\) or the \emph{short-range} profile function \(\rho(x):=p\mathbbm{1}\{0\le x\le 1\}\) for some \(p\in(0,1]\). These profile functions together with the \emph{interpolation-kernel}
    \[
        g(s,t):= (s\wedge t)^{\gamma}(s\vee t)^{\gamma'}, \text{ for } \gamma\in[0,1), \gamma'\in[0,2-\gamma),
    \]
introduced in~\cite{GraLuMo2022}, represent many of the established models such as the Poisson--Boolean model~\cite{Gouere08,Hall85,Hirsch2017} and its soft version~\cite{GGM2022}, scale-free percolation~\cite{DeijfenHofstadHooghiemstra2013,DeprezWuthrich2019} and the age-dependent random connection model~\cite{GGLM2019}, cf.~Table~\ref{tab:interPol}.
\begin{table}
\begin{center}
\caption{Various choices for \(\gamma\), \(\gamma'\) and \(\delta\) for the weight-dependent random connection model and the models they represent in the literature together with their \(\deff>2\) phases and their tail bounds \(\bar{\mu}\). Here, to shorten notation, \(\delta=\infty\) represents models constructed with \(\rho\) being the indicator function.}
\begin{tabular}{l|l|l|l}
 \textbf{Parameters} & \(\boldsymbol{\deff>2}\) \textbf{iff} & \(\boldsymbol{\bar{\mu}}\) & \textbf{Names and references} \\ \hline 
 \(\gamma = 0, \gamma' = 0, \delta = \infty\) & always & \(\infty\) & Gilbert's disc model \cite{Gilbert61}, \\ & & & random geometric graph \cite{Penrose2003} \\ \hline
 \(\gamma = 0, \gamma' =0, \delta<\infty\) & \(\delta>2\) & \(\delta-2\) & random connection model \cite{MeesterPenroseSarkar1997,Penrose2016}, \\ & & & long-range percolation \cite{Schulman1983}\\ \hline
 \(\gamma>0, \gamma'=0,\delta=\infty\) & \(\gamma<1\) & \(\tfrac{1}{\gamma}-1\) & Boolean model \cite{Hall85,Gouere08}, \\ & & & scale-free Gilbert graph \cite{Hirsch2017} \\ \hline
 \(\gamma>0, \gamma' = 0, \delta<\infty\) & \(\delta>2\), \(\gamma<\tfrac{\delta-1}{\delta}\) & \((\tfrac{\delta-1}{\gamma\delta}-1)\wedge (\delta-2)\) & soft Boolean model \cite{GGM2022} \\ \hline
 \(\gamma = 0, \gamma'>0, \delta = \infty\) & \(\gamma'<1\) & \(\infty\) & ultra-small scale-free geometric  \\ & & & network (\(\gamma'>1\)) \cite{Yukich2006}, \\ & & & weak-kernel model~\cite{GHMM2022,Lue2022} \\ \hline
 \(\gamma>0, \gamma'=\gamma, \delta\leq \infty\) & \(\delta>2\), \(\gamma<\tfrac{1}{2}\) & \(\tfrac{\delta(1-2\gamma)}{2\delta\gamma-1}\) &scale-free percolation \cite{DeijfenHofstadHooghiemstra2013,DeprezWuthrich2019}, \\ & & & geometric inhomogeneous \\ & & & random graphs \cite{BringmannKeuschLengler2019} \\ \hline
 \(\gamma>0, \gamma'=1-\gamma, \delta\leq\infty\) & never & not applicable  & age-dependent random \\ & & & connection model \cite{GGLM2019}
\end{tabular}
\label{tab:interPol}
\end{center}
\end{table}
The model~\eqref{ex_1} has also been studied as \emph{geometric inhomogeneous random graphs} and as \emph{kernel-based spatial random graphs} in a similar yet slightly different parametrisation in~\cite{KomjathyLapinskasLengler2021,BringmannKeuschLengler2019,HofstadHoornMaitra2023, JorrKomMit2023}. 

Let us calculate \(\deff\) and \(\bar{\mu}\) for the interpolation kernel model with a long-range profile function. To this end, we consider the integral in~\eqref{eq:DefDeff} and obtain
\begin{equation*}
	\begin{aligned}
		\int\limits_{n^{-1-\mu}}^1  \int\limits_{s}^1  s^{-\gamma\delta}t^{-\gamma'\delta}n^{-\delta} \, \d t \, \d s \asymp 
        \begin{cases}
		    n^{-\delta}\int\limits_{n^{-1-\mu}}^1  s^{-\gamma\delta} \, \d s, & \text{ if } \ \gamma'<\nicefrac{1}{\delta}, \\
            n^{-\delta}\int\limits_{n^{-1-\mu}}^1 s^{1-\gamma\delta-\gamma'\delta} \, \d s , & \text{ if } \ \gamma'>\nicefrac{1}{\delta}.
		\end{cases}
	\end{aligned}
\end{equation*}
There are four cases to be distinguished. If \(\gamma'<1/\delta\) and \(\gamma<1/\delta\), or \(\gamma'>1/\delta\) and \(\gamma<2/\delta-\gamma'\), the respective integral on the right-hand side is of constant order and hence simply \(\deff=\delta\). For the case \(\gamma'<1/\delta\) and \(\gamma>1/\delta\), we have
\begin{equation*}
	\begin{aligned}
		n^{-\delta}\int\limits_{n^{-1-\mu}}^1s^{-\gamma\delta} \, \d s\asymp n^{\gamma\delta-1-\delta-\mu(1-\gamma\delta)}.
	\end{aligned}
\end{equation*}
Sending \(\mu\downarrow 0\), we obtain \(\deff>2\) if \(\delta>2\) and \(\gamma<1-1/\delta\) and \(\deff<2\) if either \(\delta\leq 2\) or \(\gamma>1-1/\delta\) by continuity. To obtain \(\bar{\mu}\), we solve
\[
	\bar{\mu}=1+\delta-\gamma\delta+\bar{\mu}(1-\gamma\delta)-2\qquad\Longleftrightarrow\qquad \bar{\mu}=\frac{\delta-1}{\gamma\delta}-1.
\]
Let us finally consider the remaining case \(\gamma'>1/\delta\) and \(\gamma>2/\delta-\gamma'\). By~\cite[Proposition~2.4]{Lue2022}, we have \(\lambda_c=0\) if \(\gamma+\gamma'>1\). We therefore restrict ourselves to the case \(\gamma'\leq 1-\gamma\). Since additionally \(\gamma<1\), this further implies \(\delta>2\). Then, the integral to be calculated reads
\begin{equation*}
	\begin{aligned}
		n^{-\delta}\int\limits_{n^{-1-\mu}}^1 s^{1-\gamma\delta-\gamma'\delta} \, \d s & \asymp n^{-\delta-(1+\mu)(2-\delta(\gamma'+\gamma))}.
	\end{aligned}
\end{equation*}
Sending \(\mu\downarrow 0\), we observe \(\deff>2\) if \(\gamma'<1-\gamma\). We further solve 
\[
	\bar{\mu}=\delta+(1+\bar{\mu})(2-\delta(\gamma'+\gamma))-2\qquad\Longleftrightarrow\qquad\bar{\mu}= \frac{\delta(1-\gamma-\gamma')}{\delta(\gamma+\gamma')-1}.
\]
Summarising the above, \(\deff>2\) if \(\delta>2\), \(\gamma<1-1/\delta\), and \(\gamma'<1-\gamma\), see Figure~\ref{fig:Deff}. Results for models with a short-range profile function can be obtained by sending \(\delta\to\infty\). To further deduce tail bounds for the Euclidean diameter and the number of points of the origin's component for the examples from the literature, one chooses \(\gamma\), \(\gamma'\), and \(\delta\) appropriately, see Table~\ref{tab:interPol}. Let us further remark that for the soft Boolean model with the choice of parameters \(\gamma'=0\), \(1-1/\delta<\gamma<1-1/(\delta+1)\), and \(\delta>1\), we obtain \(\deff<2\) and therefore \(\widehat{\lambda}_c=0\). However, by~\cite{GLM2021} we still have \(\lambda_c>0\). Further, for \(\gamma>1-1/(\delta+1)\), still \(\deff<2\) but now \(\lambda_c=0\). This shows on the one hand that, unlike in the Boolean model, we do not always have \(\widehat{\lambda}_c=\lambda_c\). On the other hand, it also shows that \(\deff\) cannot be used alone to derive the absence of subcritical phases.
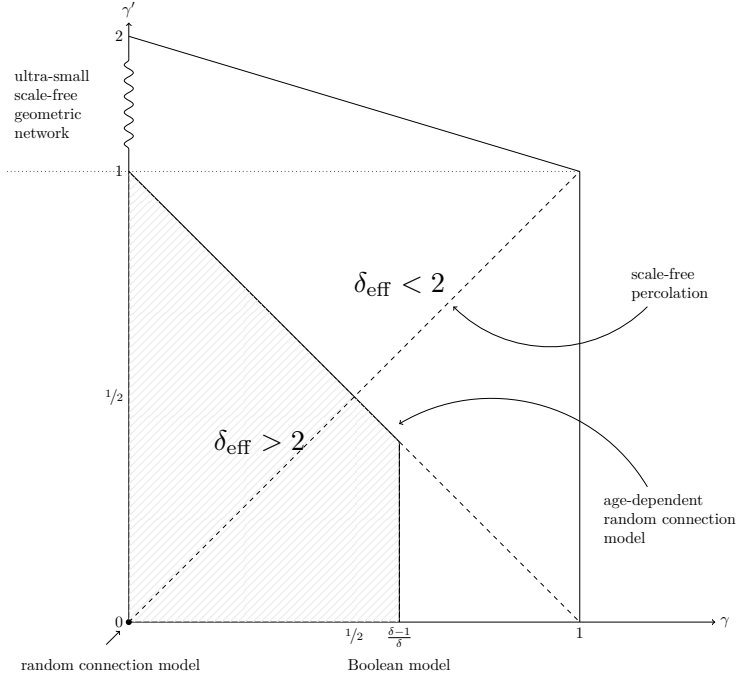
\begin{figure}
\begin{center}
\resizebox{0.6\textwidth}{!}{
\begin{tikzpicture}[every node/.style={scale=1}]
    \draw[->] (0,0) to (13,0) node[right] {$\gamma$};
    \draw	(5,0) node[anchor=north] {\nicefrac{1}{2}}
    (10,0) node[anchor=north] {1};

    \draw[dotted] (-2.7,10) to (10,10);
    \draw[] (10,0) to (10,10)
	   (10,10) to (0,13);

    \draw[](0,0) to (0,10.5);
    \draw [->] (0,12.5) to (0,13.3) node[above] {$\gamma'$};
    \draw[decorate, decoration = {snake, segment length = 10 pt, amplitude = 1mm}] (0, 10.5)--(0,12.5);
    \draw	(0,0) node[anchor=east] {0}
        (0,5) node[anchor=east] {\nicefrac{1}{2}}
        (0,10) node[anchor=east] {1}
        (0,13) node[anchor = east] {2};

    \draw[thick] (0,10) to (6,4);
    \draw[dashed] (6,4) to (10,0);
    \draw[dashed] (0,0) to (10,10);
    \draw (6,-0.7) node[align = left, anchor = north] {Boolean model};
    \draw (-1.7,11.5) node[align = left] {ultra-small \\ scale-free \\ geometric \\ network};
    \draw (12,2.3) node[align = left] {age-dependent \\ random connection \\ model};
    \draw[->, bend angle = 45, bend right] (11.5,3) to (6,4.4);
    \draw (12,7.5) node[align = left] {scale-free \\ percolation};
    \draw[->, bend angle = 45, bend left] (11.5,7) to (7.2,7);
    \draw (0,0) node[circle, fill = black, scale=0.4] {};
    \draw (-0.4,-0.7) node[align = left, anchor = north] {random connection model};
    \draw[->] (-0.5, -0.5) to (-0.2,-0.2);

    \draw (6, 0) node[anchor = north] {$\tfrac{\delta-1}{\delta}$};
    \draw[thick] (6,0) to (6,4);

    \draw[pattern=north east lines, pattern color=lightgray!30!, draw=none] 
        (0,0) plot[smooth,samples=200,domain=0:6,variable=\g]({0},{10*\g/6}) -- 
         plot[smooth,samples=200,domain=6:0,variable=\g]({6},{4*\g/6});
 
    \draw (6,7.5) node[scale = 1.8, thick] {$\deff<2$};
    \draw (2.9,4) node[scale = 1.8, thick, align = left] {$\deff>2$};
\end{tikzpicture}
}
\end{center}
\caption{Phase diagram in \(\gamma\) and \(\gamma'\) for the weight-dependent random connection model constructed with the interpolation kernel and a profile function of polynomial decay at rate \(\delta>2\). The solid lines marks the phase transition \(\deff=2\). Dashed lines represent no change of behaviour.}
\label{fig:Deff}
\end{figure}

\paragraph{Soft Boolean model with local interference} We also present an example of a mixing graph where the edge probabilities actually depend on the surrounding point cloud. The idea is to combine the soft Boolean model~\cite{GGM2022} with local inference and noise in the spirit of SINR percolation~\cite{DousseEtAl2006, Tobias2020}. To formulate the model let us denote, for a given vertex \(\y=(y,s)\) and \(\xi\geq 0\), the random variable
\[
   N^\xi(\y,\cX):= \sharp \big\{\z\in\cX : |z-y|^d\leq s^{-\xi}\big\}. 
\]
The graph \(\G=\G^\lambda\) is then generated by connecting \(\x=(x,t)\) and \(\y=(y,s)\) with probability
\begin{align}\label{ex_2}
   \mathbf{p}(\x,\y,\cX\setminus\{\x,\y\}) = \1\{s<t\}\frac{1\wedge s^{-\gamma\delta}|x-y|^{-d\delta}}{1+N^{\xi}(\y,\cX\setminus \{\x,\y\})}+\1\{s\geq t\}\frac{1\wedge t^{-\gamma\delta}|x-y|^{-d\delta}}{1+N^{\xi}(\x,\cX\setminus \{\x,\y\})},
\end{align}
where again \(\gamma\in[0,1)\) and \(\delta>2\).
Since \(N_\lambda^\xi((y,s),\cX)\overset{d}{\sim}N_\lambda^\xi((o,s),\cX)\), Condition~\eqref{eq:varphi} is satisfied with 
\[
   \varphi(s,t,r)= \1\{s<t\}\E^\lambda\Big[\frac{1\wedge s^{-\gamma\delta}r^{-\delta}}{1+N^{\xi}((o,s),\cX)}\Big]+\1\{s\geq t\}\E^{\lambda}\Big[\frac{1\wedge t^{-\gamma\delta}r^{-\delta}}{1+N^{\xi}((o,t),\cX)}\Big].
\]

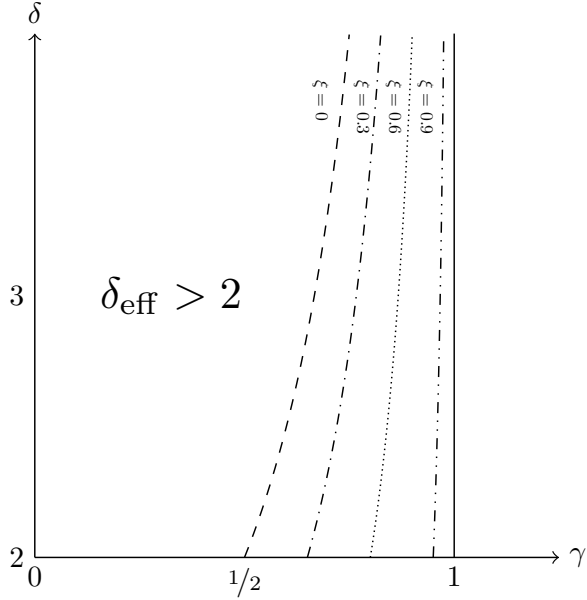
\begin{figure}
    \begin{center}
    \resizebox{0.5\textwidth}{!}{
    \begin{tikzpicture}[every node/.style={scale=0.8}]
        \draw[->] (0,0) -- (5,0) node[right] {$\gamma$};
        \draw	(0,0) node[anchor=north] {0}
        (2,0) node[anchor=north] {\nicefrac{1}{2}}
        (4,0) node[anchor=north] {1};

        \draw[->] (0,0) -- (0,5) node[above] {$\delta$};
        \draw (0,0) node[anchor=east] {2}
            (0,2.5) node[anchor=east] {3};

        \draw[] (4,0) -- (4,5);

        \draw (1.3,2.5) node[scale = 1.5, thick]{$\deff>2$};
        
        \draw[dashed] (0,0) plot[domain=0.5:3/4,variable=\g]({4*\g},{-5+2.5*1/(1-\g)});
        \draw (90/33, 4.3) node[scale = 0.6, rotate = 270] {$\xi =0\phantom{.0}$};

        \draw[dash dot] (0,0) plot[domain=13/20:33/40,variable=\g]({4*\g},{-5+2.5*((1-0.3)/(1-\g)});
        \draw (103/33, 4.3) node[scale = 0.6, rotate = 270] {$\xi =0.3$};

        \draw[densely dotted] (0,0) plot[domain=16/20:36/40,variable=\g]({4*\g},{-5+2.5*((1-0.6)/(1-\g)});
        \draw (113/33, 4.3) node[scale = 0.6, rotate = 270] {$\xi =0.6$};

        \draw[dash dot dot] (0,0) plot[domain=19/20:39/40,variable=\g]({4*\g},{-5+2.5*((1-0.9)/(1-\g)});
        \draw (123/33, 4.3) node[scale = 0.6, rotate = 270] {$\xi =0.9$};
        
    \end{tikzpicture}
    }
    \end{center}
    \caption{Phase diagram for \(\gamma\) and \(\delta\) for the soft Boolean model with local interference. Represented from left to right the phase transition for \(\deff>2\) for \(\xi=0,0.3,0.6,0.9\).}
    \label{fig:DistractBool}
\end{figure}

Let us note that the model~\eqref{ex_2} is a combination of the soft Boolean model, a special instance of the weight-dependent random connection model above, with random interference coming from the vertices surrounding the vertices that are to be connected. More precisely, each vertex \(\y=(y,s)\) has a sphere of influence of radius \(s^{-\gamma/d}\) and a sphere of interference of radius \(s^{-\xi/d}\). The mark \(s\) can be understood as an inverse attraction parameter and the smaller \(s\) the more attractive the vertex is.  Now, the vertex \(\y\) likes to connect to each vertex located within its sphere of influence, enlarged by independent Pareto random variables to include long-range effects (cf.\ the description of the soft Boolean model in~\cite{GGM2022}). However, \(\y\) gets distracted by all vertices contained in its sphere of interference which makes it more difficult to form edges. Note that either both spheres are big (for small \(s\)) or small (for large \(s\)).
For \(\xi=0\), each sphere of interference is of radius one and the model reduces to a version of the soft Boolean model with some additional yet on large scales insignificant fluctuations. 

We start by showing that \(\G\) is mixing for \(\xi<1\). To this end, observe that the left-hand side in~\eqref{eq:mixing} is bounded by some constant times the probability of the complement of the largest event on which the covariance is zero. Further observe that \(G(\alpha)\) and \(G(\alpha,x)\) are independent whenever there is no pair of vertices in the involved balls such that their spheres of interference intersect. In other words, the covariance in \eqref{eq:mixing} is bounded by the probability of existence of vertices \(\y\in\cX(\mathcal{B}({3\alpha^{1/d}}))\) and \(\z\in\cX(\mathcal{B}({3\alpha^{1/d}},x))\) for which
\[
    u_y^{-\xi}+u_z^{-\xi}\geq |z-y|^d.
\]
We recall that \(|x|^d\geq 6^d\alpha\) and hence \(|z-y|^d\geq c\alpha\) for some constant \(c\). Moreover, \(u_y^{-\xi}+u_z^{-\xi}<2(u_y^{-\xi}\vee u_z^{-\xi})\). Hence, we can further bound by the probability of existence of a vertex located in one of the two involved balls with mark smaller than \(c\alpha^{-\nicefrac{1}{\xi}}\). By Poisson thinning~\cite{LastPenrose2017}, the expected number of such vertices and hence also the probability of existence of at least one such vertex is bounded by
\[
    C \lambda \alpha^{1-1/\xi}
\]
for some constant $C>0$. 
\begin{figure}
    \centering
    \begin{subfigure}{0.4\textwidth}
        \centering
        \frame{\includegraphics[scale = 0.491]{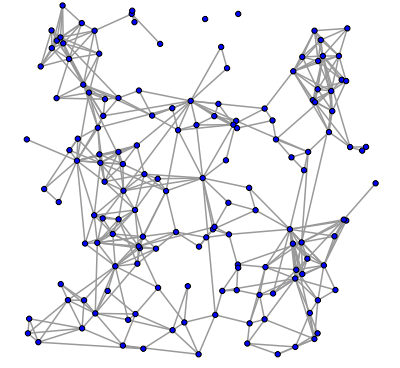}}
        \caption{Soft Boolean model} 
        \label{fig:softVsInterSoft}
    \end{subfigure}\qquad
    \begin{subfigure}{0.4\textwidth}
        \centering
        \frame{\includegraphics[scale = 0.49]{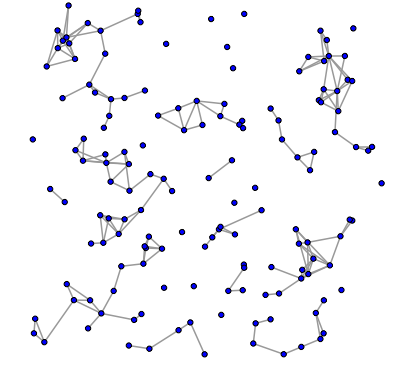}}
        \caption{Soft Boolean model with local interference}
        \label{fig:softVsInterInter}
    \end{subfigure}
    \caption{Examples for the soft Boolean model, Fig.~\ref{fig:softVsInterSoft}, and the soft Boolean model with local interference, Fig.~\ref{fig:softVsInterInter}, on the same \(150\) vertices sampled from a Poisson process of intensity \(\lambda=0.04\). For the edge probabilities, the parameters \(\gamma=0.65, \delta=2.7\), and \(\xi=0.3\) are used. Hence, \(\deff<2\) for the soft Boolean model but \(\deff>2\) for the model with local interference.}
    \label{fig:softVsInter}
\end{figure}
Now, since \(\xi<1\), the graph is mixing with exponent \(\zeta=1/\xi-1\). Let us further calculate some values of \(\deff\), see Figure~\ref{fig:DistractBool}. Since \(N^\xi((o,s),\cX)\) is Poisson distributed with parameter of order \(s^{-\xi}\), we infer
\begin{equation*}
    \begin{aligned}
        n^{-\delta}\int\limits_{n^{-1-\mu}}^1 s^{-\gamma\delta} \E^\lambda\Big[\frac{1}{1+ N^\xi((o,s),\cX)}\Big] \d s &\asymp n^{-\delta}\int\limits_{n^{-\mu-1}}^1 s^{-\gamma\delta+\xi} \, \d s  \asymp n^{-\delta}\vee n^{-\delta - (1
        +\mu)(1-\gamma\delta+\xi)}.   
    \end{aligned}
\end{equation*}
Hence, sending \(\mu\downarrow 0\), we obtain if \(\gamma<{(1+\xi)}/{\delta}\), then \(\deff=\delta\). This is in particular always true if \(\xi>\delta-2\). If \(\gamma>{(1+\xi)}/{\delta}\), we have
\begin{equation*}
    \begin{aligned}
        \deff>2 & \qquad \Longleftrightarrow\qquad \delta(1-\gamma)+\xi>1\qquad \Longleftrightarrow \qquad\gamma<(\delta+\xi-1)/\delta. 
    \end{aligned}
\end{equation*}
For \(\xi=0\), we recover the bound for the soft Boolean model found in the previous paragraph. For \(\xi>0\), we observe that the local interferences indeed make it harder to percolate, cf.~Figure~\ref{fig:softVsInter}. 
To get bounds on the tail distribution of \(\Mcal\) and \(\Ncal\), we calculate for \((1+\xi)/\delta<\gamma<(\delta+\xi-1)/\delta\)
\begin{equation*}
	\begin{aligned}
		\bar{\mu} = \delta+(1+\bar{\mu})(1-\gamma\delta\xi)-2 \qquad\Longleftrightarrow\qquad \bar{\mu} = (\delta(1-\gamma)+\xi-1)/(\gamma\delta-\xi).
	\end{aligned}
\end{equation*} 
We therefore obtain the following tail bound for the Euclidean diameter and the number of points
\[
	\big(\tfrac{1}{\xi}-1\big)\wedge 
	\begin{cases}
		(\delta-2), & \text{ if } \ \gamma\leq\tfrac{1+\xi}{\delta}, \\
		\frac{\delta(1-\gamma)+\xi-1}{\gamma\delta-\xi}, & \text{ if } \ \tfrac{1+\xi}{\delta}<\gamma<\tfrac{\delta(1-\gamma)+\xi-1}{\gamma\delta-\xi}.
	\end{cases}
\]

\section{Proofs}\label{sec:Proofs}
In this section, we prove our Theorems~\ref{thm:Subcritical} and~\ref{thm:NonSharpness}. The first section is devoted to prove Theorem~\ref{thm:Subcritical}. Theorem~\ref{thm:NonSharpness} is proven afterwards in Section~\ref{sec:proof(ii)}.

\subsection{Proof of Theorem~\ref{thm:Subcritical}}\label{sec:proof(i)} 
Recall the notation of  
\[
    \cX(D\times I)=\{\x=(x,u_x)\in\cX\colon x\in D, u_x\in I\}
\] 
and \(\cX(D)=\cX(D\times(0,1))\). Further recall the event \(G(\alpha,x)\), resp.\ \(G(\alpha)\) introduced in~\eqref{eq:GEvent}.

Let us write $A^c=\R^d\setminus A$, for the complement of any $A\subset \R^d$, and define two events describing the occurrence of long edges,
\begin{equation*}
    \begin{aligned}
        H(\alpha) &=\big\{\exists \x\in \cX\big(\mathcal{B}({3\alpha^{1/d}})^c\big) \text{ and } \y\in \cX(\mathcal{B}({2\alpha^{1/d}}))\colon \x\sim \y\big\} \quad \text{ and}\\
        F(\alpha) &=\{\exists\x,\y\in \cX(\mathcal{B}({20\alpha^{1/d}})): \x\sim \y\text{ and }|x-y|^d\geq \alpha\}.
    \end{aligned}
\end{equation*}
 
We start by proving that both \(\P^\lambda(H(\alpha))\) and \(\P^\lambda(F(\alpha))\) tend to zero as \(\alpha\to\infty\) whenever \(\delta_\text{eff}>2\). To this end, we introduce for \(\varepsilon>0\) 
\[
	\bar{\mu}(\varepsilon):=\sup\{\mu>0: \mu<\psi^-(\mu)-2-\varepsilon\}.
\]

\medskip

\begin{lemma}\label{lem:Errors} If \(\deff^- >2\), then for all \(\varepsilon\in(0,\deff^- -2)\) there exists \(A>1\) such that for all \(\mu<\bar{\mu}(\varepsilon)\) there exists a constant \(C=C(\mu)>0\) such that for all \(\lambda>0\) and \(\alpha>A\), we have
        \begin{equation*}
            \P^\lambda\big(H(\alpha)\big)\vee \P^\lambda\big(F(\alpha)\big) \leq C (\lambda\vee\lambda^2) \alpha^{-\mu}.
        \end{equation*}
\end{lemma}

\medskip 

\begin{proof}
    Define for \(\mu\in(0,\bar{\mu}(\varepsilon))\) the event
    \[
        \widetilde{H}(\alpha) := \widetilde{H}(\alpha,\mu)= \Big\{\exists \x\in \cX\big(\mathcal{B}({3\alpha^{1/d}})^c\big) \text{ and } \y\in \cX\big(\mathcal{B}({2\alpha^{1/d}})\big)\colon u_x\wedge u_y\geq |x|^{-d(1+\mu)} \text{ and }\x\sim \y\Big\},
    \]
    and note that 
    \[H(\alpha)\subset\widetilde{H}(\alpha)\cup\Big\{\exists\x\in\cX\big(\mathcal{B}\big({3\alpha^{1/d}}\big)^c\big)\colon |u_x|<|x|^{-d(1+\mu)}\Big\}\cup\Big\{\exists\y\in\cX\big(\mathcal{B}({2\alpha^{1/d}}\big)\colon |u_y|<3^d\alpha^{-(1+\mu)}\Big\}.\]
    By the thinning theorem for Poisson processes~\cite{LastPenrose2017}, the latter two events have probabilities of order
    \begin{equation}\label{eq:markCut}
    	\begin{aligned}
    		\lambda \int\limits_{|x|^d>3^d\alpha} \d x \ |x|^{-d(1+\mu)} = \tfrac{\omega_d}{\mu 3^{d\mu}} \lambda \alpha^{-\mu},\qquad\text{ and }\qquad \lambda\int\limits_{|y|^d<2^d\alpha} \d y \ (3^d\alpha)^{-(1+\mu)} = \tfrac{3^d\omega_d}{2^{d(1+\mu)}} \lambda\alpha^{-\mu},
    	\end{aligned}
    \end{equation}
    where \(\omega_d\) denotes the volume of the \(d\)-dimensional unit ball. For the event \(\widetilde{H}(\alpha)\), we infer using Mecke's equation~\cite{LastPenrose2017}, \eqref{eq:varphi} and the definition of \(\psi^-(\mu)\), 
    \begin{equation*}
        \begin{aligned}
            \P(\widetilde{H}(\alpha)) & \leq \E^\lambda\Big[\sum_{\y\in\cX\big(\mathcal{B}({2\alpha^{1/d}})\big)}\sum_{\x\in\cX\big(\mathcal{B}({3\alpha^{1/d}}^c)\big)}\1\{u_x,u_y>|x|^{-d(1+\mu)}\}\mathbf{p}(\x,\y,\cX\setminus\{\x,\y\})\Big] \\
            & = \lambda^2\int\limits_{|y|^d<2^d\alpha}\d y \int\limits_{|x|^d>3^d\alpha}\d x\int\limits_{|x|^{-d(1+\mu)}}\d u_y \int\limits_{|x|^{-d(1+\mu)}}\d u_x \; \E^\lambda[\mathbf{p}(\x,\y,\cX)] \\ 
            & = \lambda^2\int\limits_{|y|^d<2^d\alpha}\d y \int\limits_{|x|^d>3^d\alpha} \d x \int\limits_{|x|^{-d(1+\mu)}}^{1}\d u_y \int\limits_{|x|^{-d(1+\mu)}}^1\d u_x \;  \varphi(u_x,u_y,|x-y|^d) \\
            & \leq C \lambda^2 \omega_d \alpha \int\limits_{|x|^d> \alpha} \d x \int\limits_{|x|^{-d(1+\mu)}}^{1}\d u_y \int\limits_{|x|^{-d(1+\mu)}}^1\d u_x \ \varphi(u_x,u_y,|x|^d)\\
            & \leq C\lambda^2 \omega_d \alpha\int\limits_{|x|^d> \alpha} |x|^{-d(\psi^-(\mu)+\varepsilon)} \d x \\
            &\leq C\omega_d^2 \lambda^2 \alpha^{2-\psi^-(\mu)+\varepsilon}
        \end{aligned}
    \end{equation*}
    for sufficiently large \(\alpha\) and where we have used \(\psi^-(\mu)>1\) for all \(\mu<\bar{\mu}\) in the last step to bound the integration constant \(\omega_d/(\psi^-(\mu)-1)<\omega_d\).  The claim follows then since \(\mu<\psi^-(\mu)-2-\varepsilon\) for all \(\mu<\bar{\mu}(\varepsilon)\). 
    
    Analogously, we can again use the same cut-off of the vertex marks and consider the coinciding event \(\widetilde{F}_\alpha\) and have
    \begin{equation*}
        \begin{aligned}
            \P(\widetilde{F}(\alpha)) & \leq \lambda^2\int\limits_{|y|^d<20^d\alpha}\d y \int\limits_{|x-y|^d>\alpha} \d x \int\limits_{|x|^{-d(1+\mu)}}^{1}\d u_y \int\limits_{|x|^{-d(1+\mu)}}^1\d u_x \ \varphi(s,t,|x-y|^d) \\ 
            & \leq C\lambda^2 \omega_d \alpha \int\limits_{|x|^d>\alpha} \d x \int\limits_{|x|^{-d(1+\mu)}}^{1}\d u_y \int\limits_{|x|^{-d(1+\mu)}}^1\d u_x \ \varphi(s,t,|x|^d) \\ 
            &\leq C \lambda^2 \alpha^{2-\psi^-(\mu)+\varepsilon},
        \end{aligned}
    \end{equation*}
    as desired.
\end{proof}

\medskip

\begin{proof}[Proof of Theorem~\ref{thm:Subcritical}]
We start proving \(\widehat{\lambda}_c>0\) whenever \(\G\) is mixing with index \(\zeta>0\) and \(\deff>2\). We fix some \(\varepsilon\in(0,\deff-2)\) and the corresponding \(A>1\) from Lemman~\ref{lem:Errors} throughout the first part of the proof. Then, for all \(\alpha>A\)
\begin{equation*}
	\begin{aligned}
		\P^\lambda\Big(\exists \x\in\cX\big(\B({\alpha^{1/d}})\big),\y\in\cX\big(\B(3\alpha^{1/d})^c\big):\x\leftrightarrow\y\Big) & \leq \P^\lambda(G(\alpha))+\P^{\lambda}(H(\alpha))  \leq \P^\lambda(G(\alpha))+C\lambda \alpha^{-\mu}
	\end{aligned}
\end{equation*}
for any \(\mu<\bar{\mu}(\varepsilon)\) and \(\lambda\in(0,1)\). It therefore remains to show that \(\P^\lambda(G(\alpha))\to 0\) for sufficiently small \(\lambda\in(0,1)\). To this end, we apply the multi-scale scheme of~\cite{Gouere08}. Let us denote by \(\partial B \) the boundary of a set \(B\subset\R^d\). Let further \(\mathcal{K},\mathcal{L}\subset\R^d\) be two finite sets satisfying \(\mathcal{K}\subset\partial\B(9)\) and \(\mathcal{L}\subset\partial\B(20)\) such that 
\[
	\partial\B(9)\subset\bigcup_{k\in\mathcal{K}}\B(1,k), \quad \text{ and } \quad \partial\B(20)\subset\bigcup_{l\in\mathcal{L}}\B(1,l).
\]
The key observation is 
\begin{equation}\label{eq:FactoriseP(G)}
	    G({10^d\alpha})\setminus F(\alpha)\subset \Big(\bigcup_{k\in \mathcal{K}}G(\alpha,\alpha^{1/d} k)\Big)\cap\Big(\bigcup_{{l}\in \mathcal{L}}G(\alpha,\alpha^{1/d}{l})\Big).
\end{equation}
This is due to the fact that on \(G({10^d\alpha})\setminus F(\alpha)\), there exists a path from a vertex located in \(\mathcal{B}({10\alpha^{1/d}})\) to some vertex located in the annulus \(\mathcal{B}({30\alpha^{1/d}})\setminus\mathcal{B}({20\alpha^{1/d}})\) using only vertices located in \(\mathcal{B}({30\alpha^{1/d}})\) and edges no longer than \(\alpha^{1/d}\). Therefore, at least one vertex of the path must lie within distance \(\alpha^{1/d}\) of the sphere \(\partial\B({9\alpha^{1/d}})\) and another vertex within the same distance of the sphere \(\partial\B(20\alpha^{1/d})\). By the covering property of \(\mathcal{K}\), there exists some \(k\in\mathcal{K}\) such that the vertex within distance \(\alpha^{1/d}\) of \(\partial\B(9\alpha^{1/d})\) is located in \(\mathcal{B}({\alpha^{1/d}},\alpha k)\).  As the (sub)path starting in that vertex contains only edges shorter than \(\alpha^{1/d}\) and reaches a vertex located in \(\B(20\alpha^{1/d})^c\), the starting vertex is connected to some vertex in \(\B(3\alpha^{1/d},x_k)\setminus\B(2\alpha^{1/d},x)\) by a path only consisting of vertices located in \(\B(3\alpha^{1/d},x)\). Put differently, the event \(G(\alpha, \alpha^{1/d}k)\) occurs. By a similar argumentation also \(G(\alpha,\alpha^{1/d}l)\) occurs for some \(l\in\mathcal{L}\), see Figure~\ref{fig:sketch}. As a result
\[
	\P^\lambda\big(G(10^d\alpha)\setminus F(\alpha)\big)\leq \sum_{k\in\mathcal{K}, \, l\in\mathcal{L}}\P^\lambda\big(G(\alpha,\alpha^{1/d}k)\cap G(\alpha,\alpha^{1/d}l)\big).
\]
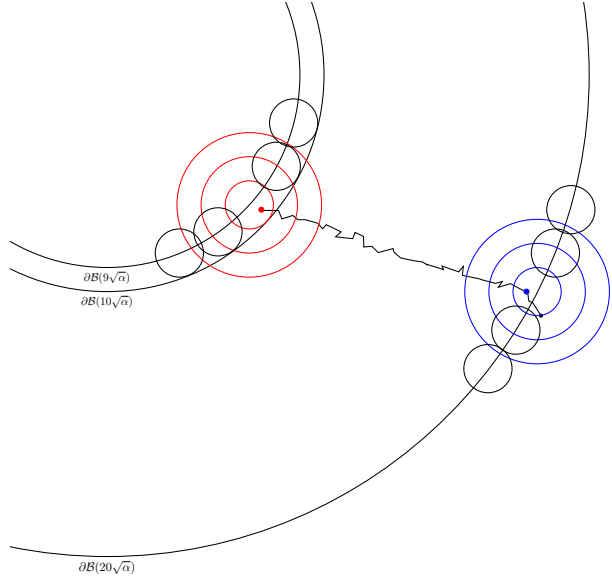
\begin{figure}
    \begin{center}
    \resizebox{1\textwidth}{!}{
    \begin{minipage}{\textwidth}
    \centering
    \begin{tikzpicture}[scale = 0.65, every node/.style={scale=0.4}]
        \clip (-2,1.5) rectangle ++(13,-13);
        \draw[] (0,0) circle(4);
        \draw[] (0,0) circle(4.5);
        \draw[] (0,0) circle(10);

        \draw (0,-4) node[anchor = north, scale = 1.1]{$\partial\B(9\sqrt{\alpha})$};
        \draw (0,-4.5) node[anchor = north, scale = 1.1]{$\partial\B(10\sqrt{\alpha})$};
        \draw (0,-10) node[anchor = north, scale = 1.2]{$\partial  \B(20\sqrt{\alpha})$};

        \node (S) at (3.2,-2.8)[scale = 0.5, circle, fill = red, color =red,  label = {},draw]{};
        \node (T1) at (8.7, -4.5)[draw, circle, scale = 0.5, color=blue, fill = blue]{};
        \node (T2) at (9,-5)[draw, circle, scale = 0.3, color =black, fill = black]{};
        \draw[decorate, decoration = {random steps, segment length = 1 mm}] (S)--(T1);
        \draw[decorate, decoration = {random steps, segment length = 1 mm}] (T1)--(T2);
        \draw[red] (2.95,-2.7) circle(0.5);
        \draw[red] (2.95,-2.7) circle(1);
        \draw[red] (2.95,-2.7) circle (1.5);
        \draw (3.51,-1.9) circle(0.5);
        \draw (2.3,-3.26) circle(0.5);
        \draw (3.87,-1) circle(0.5);
        \draw (1.5,-3.7) circle(0.5);
        \draw[blue] (8.92,-4.5) circle(0.5);
        \draw[blue] (8.92,-4.5) circle(1);
        \draw[blue] (8.92,-4.5) circle(1.5);
        \draw (8.48, -5.3) circle(0.5);
        \draw (9.3, -3.7) circle(0.5);
        \draw (7.9,-6.1) circle(0.5);
        \draw (9.62,-2.8) circle(0.5);
    \end{tikzpicture}
    \end{minipage}
    }
    \end{center}
    \caption{Sketch of~\eqref{eq:FactoriseP(G)} in \(d=2\). A path starting from a vertex located inside \(B(10\sqrt{\alpha})\) in red leaving \(\B(20\sqrt{\alpha})\). On the spheres \(\partial\B(9{\sqrt{\alpha}})\) and \(\partial\B(20\sqrt{\alpha})\) parts of the covering \(\bigcup_{k\in\mathcal{K}}\B(\alpha,\sqrt{\alpha}k)\) and \(\bigcup_{l\in\mathcal{L}}\B(\alpha,\sqrt{\alpha}l)\) are shown. In red, the ball of the covering the starting vertex is located in. Also, the balls of radius \(2\sqrt{\alpha}\) and \(3\sqrt{\alpha}\) around the centre of the red covering ball are shown. Since all edges are shorter than \(\sqrt{\alpha}\) one point of the path must be located in \(\B(3\sqrt{\alpha})\setminus\B(2\sqrt{\alpha})\). The same situation occurs where the sphere \(\partial\B(20\sqrt{\alpha})\) is crossed and is here shown in blue.}
    \label{fig:sketch}
\end{figure}
Therefore, we obtain the existence of a constant \(C_1>0\) depending only on the choice of \(\mathcal{K}\) and \(\mathcal{L}\) such that
\[
    \P^\lambda(G({10^d\alpha}))\leq C_1 \ \P^\lambda(G(\alpha))^2 + \P^\lambda(F(\alpha)) + C_1\max_{k\in\mathcal{K},\, l\in\mathcal{L}}\Big(\P^\lambda\big(G(\alpha,\alpha^{1/d} k)\cap G({\alpha},(\alpha^{1/d} l)\big)-\P^\lambda(G(\alpha))^2\Big).
\]
Since our graph is mixing with index \(\zeta\), we derive from~\eqref{eq:mixing} and translation invariance
\[
    \max_{k\in \mathcal{K},\, l \in\mathcal{L}}\Big(\P^\lambda\big(G(\alpha,\alpha^{1/d} k)\cap G({\alpha},\alpha^{1/d} l)\big)-\P^\lambda(G(\alpha))^2\Big) \leq  C_{\text{mix}}\lambda\alpha^{-\zeta}.
\]
Combined with Lemma~\ref{lem:Errors}, this yields for \(\mu<\bar{\mu}(\varepsilon)\) the existence of a constant \(C_2>0\) such that for all \(\alpha>A\)
\begin{equation*}
    \P^\lambda(G({10^d\alpha}))\leq C_2 \P^\lambda(G(\alpha))^2 + C_2\lambda \alpha^{-(\mu\wedge \zeta)},
\end{equation*}
from which we immediately derive for all \(\alpha>10^d A\) that
\begin{equation}\label{eq:P(G)bound}
    C_2\P^\lambda(G({\alpha}))\leq \big(C_2 \P^\lambda(G(\alpha/10^d))\big)^2 + C_2\lambda 10^{d(\mu\wedge\zeta)}\alpha^{-(\mu\wedge \zeta)}.
\end{equation}
Choose now
\[
	\lambda\leq\frac{1}{2\cdot 30^d \cdot A C_2\omega_d}\wedge\frac{1}{4C_2 10^{d(\mu\wedge\zeta)}}.
\]
Then, we have 
\[
	C_2\P^\lambda(G(\alpha))\leq C_2\E^\lambda[\sharp\cX(\B(3\alpha^{1/d}))]=C_2\lambda \omega_d 3^d  \alpha \leq 1/2, \qquad \forall \, \alpha\in[A,10^d A], 
\]
as well as 
\[
	C_2 10^{d(\mu\wedge\zeta)} \lambda \alpha^{-(\mu\wedge\zeta)}\leq 1/4, \qquad \forall \, \alpha\geq A.
\]
Therefore, \(C_2\P^\lambda(G(\alpha))\leq 1/2\) also for all \(\alpha\geq 10^d A\) by~\eqref{eq:P(G)bound}. Let us write \(g(\alpha):=C_2\lambda 10^{d(\mu\wedge\zeta)}\alpha^{-\mu\wedge\zeta}\). We then infer for all \(\alpha'\in[A,10^d A]\), and \(n\in\N\) by iterating \eqref{eq:P(G)bound}
\begin{equation*}
	\begin{aligned}
		C_2\P^\lambda\big(G(10^{dn}\alpha')\big) & \leq \tfrac{C_2}{2}\P^\lambda\big(G(10^{d(n-1)}\alpha')\big)+g(10^{dn}\alpha') \leq 2^{-n-1}+\sum_{j=1}^n \frac{g(10^{dj}\alpha')}{2^{n-j}}. 
	\end{aligned}
\end{equation*}

We obtain \(\sum_{j=1}^n g(10^{dj}\alpha')2^{-(n-j)}\to 0\) as \(n\to\infty\) by Cauchy's product formula as \(\sum_n 2^{-n}<\infty\) as well as \(\sum_n g(10^{dn}\alpha')<\infty\), proving Part~\textit{(i)} of Theorem~\ref{thm:Subcritical}. 

To prove Part~\textit{(ii)}, we fix \(\lambda<\widehat{\lambda}_c\). We aim for showing the finiteness of the integral
\[
	\E^\lambda \Mcal^a \leq 1 + \int_1^\infty \alpha^{a-1}\P^\lambda(\Mcal\geq \alpha) \, \d \alpha
\]
for any \(a<\zeta\wedge\bar{\mu}\). We hence have to bound the tail probability of \(\Mcal\) appropriately. For fixed \(a<\zeta\wedge\bar{\mu}\) we choose \(\varepsilon>0\) with corresponding \(A>1\) such that \(a<\zeta\wedge \bar{\mu}(\varepsilon)\) and fix some \(\mu\in(a,\bar{\mu}(\varepsilon))\). 

We start by bounding the probability of the event \(E(\alpha)\) that the origin itself is incident to an edge no shorter than \(\alpha^{1/d}\). With the same arguments used in~\eqref{eq:markCut}, we can restrict ourselves to neighbours \(\x=(x,u_x)\) with \(u_x>|x|^{-d(1+\mu)}\) and \(u_o>|x|^{-d(1+\mu)}\). Then
\begin{equation*}
	\begin{aligned}
		\P_o(E(\alpha)) & \leq C\alpha^{-\mu}+\int\limits_{|x|^d>\alpha}\d x\int\limits_{|x|^{-d(1+\mu)}}^1 \d u_x \int\limits_{|x|^{-d(1+\mu)}}^1 \d u_o \, \varphi(u_x,u_o,|x|^d)\leq  C\lambda \alpha^{-\mu} + C\lambda \alpha^{1-\psi(\mu)}\leq C\lambda \alpha^{-\mu},
	\end{aligned}
\end{equation*}
by our choice of \(\mu\) for all \(\alpha>A\). This yields for the Euclidean diameter
\[
	\P_o^\lambda(\Mcal>2^d \alpha)\leq \P^\lambda(G(\alpha))+\P_o(E(\alpha))+\P(H(\alpha))\leq \P^\lambda(G(\alpha))+C\lambda \alpha^{-\mu}
\]
for all \(\alpha> A\) by the above and Lemma~\ref{lem:Errors}. Hence,
\[
	\int_A^\infty \alpha^{a-1}\P^\lambda(\Mcal>2^d\alpha) \, \d\alpha \leq \int_A^\infty \alpha^{a-1} \P^\lambda(G(\alpha)) \, \d\alpha + C\int_A^\infty \alpha^{a-\mu-1} \, \d\alpha,
\]
where the second integral of the right-hand side is finite as \(a<\mu\). It hence remains to bound the right-hand side's first integral. Recall~\eqref{eq:P(G)bound} and the function \(g(\alpha)\). Since \(\lambda<\hat{\lambda}_c\), we have \(\P^\lambda(G(\alpha))\to 0\). Thus, by increasing \(A\) if necessary, we find \(\P^\lambda(G(\alpha))<10^{-da}(C_2/2)\) for all \(\alpha>A/10^d\). Then, following the arguments of~\cite{Gouere08}, we infer for \(B>A\) using~\eqref{eq:P(G)bound}
\begin{equation*}
	\begin{aligned}
		C_2 \int_A^B \alpha^{a-1}\P^\lambda(G(\alpha)) \d \alpha  & \leq C_2^2\int_A^B \alpha^{a-1}\P^\lambda(G(\alpha/10^d))^2 \, \d\alpha + \int_A^B \alpha^{a-1}g(\alpha)\, \d\alpha \\
		& \leq \tfrac{C_2}{2} \int_{A/10^d}^{B/10^d} \alpha^{a-1}\P^\lambda(G(\alpha)) \, \d\alpha + \int_A^\infty \alpha^{a-1}g(\alpha) \, \d \alpha \\
		& \leq \tfrac{C_2}{2} \int_{A/10^d}^A \alpha^{a-1}\P^\lambda(G(\alpha)) \, \d\alpha +\tfrac{C_2}{2} \int_A^B\alpha^{a-1}\P^\lambda(G(\alpha)) \, \d\alpha + \int_{A}^\infty \alpha^{a-1}g(\alpha) \, \d \alpha,
	\end{aligned}
\end{equation*}
where we used the change of variables \(\alpha\mapsto \alpha/10^d\) and the bound \(\P^\lambda(G(\alpha))<10^{-da}(C_2/2)\) in the second step. We derive, sending \(B\to\infty\),
\[
	\tfrac{C_2}{2}\int_A^\infty \alpha^{a-1}\P^\lambda(G(\alpha)) \, \d\alpha <\infty,
\] 
proving the finiteness of \(\E^\lambda \Mcal^a\). To finish the proof of Part~\textit{(ii)}, it remains to further show \(\E^\lambda \Ncal^a<\infty\). To this end, observe that for any constant \(c>0\), we have 
\[
	\P^\lambda_o(\Ncal\geq\alpha)\leq \P_o^\lambda(\Mcal\geq c\alpha)+\P^\lambda\big(\sharp\cX(B(c\alpha^{1/d}))>\alpha \big),
\]
and the proof concludes due to the exponential decay of the second term on the right-hand side observed by a standard Chernoff bound for Poisson distributions when \(c\) is chosen small enough.
\end{proof}

\subsection{Proof of Theorem~\ref{thm:NonSharpness}} \label{sec:proof(ii)}
Let us now turn to the proof of \(\deff<2\) implying \(\widehat{\lambda}_c=0\) whenever the connection probability is independent from the surrounding vertex set in which case we have
\[
	\mathbf{p}(\x,\y)=\varphi(u_x,u_y,|x-y|^d).
\] 
due to~\eqref{eq:varphi}. Let us define 
\[
	G'(\alpha)= \{\exists \x\in\cX\big(\B({\alpha^{1/d}})\big),\y\in\cX\big(\B(3^d\alpha)\setminus\B(2\alpha^{1/d})\big):\x\sim\y\}.
\] 
Theorem~\ref{thm:NonSharpness} is a direct consequence of the following lemma which is an adaption of a result of~\cite{GraLuMo2022}.

\medskip

\begin{lemma}
	Let \(\mathbf{p}(\x,\y)=\varphi(u_x,u_y,|x-y|^d)\) and \(\deff^+<2\), then for all \(\lambda>0\), we have
	\[
		\lim_{\alpha\to\infty}\P^{\lambda}(G'(\alpha))=1.	
	\]
\end{lemma}
\begin{proof}
	Fix \(\mu<1\) small enough to guarantee \(\psi^+(\mu)<2\). We consider the two random sets of vertices \(\eta_1:=\cX(\B(\alpha^{1/d}))\) and \(\eta_2:=\cX(B(3\alpha^{1/d})\setminus\B(2\alpha^{1/d}))\). By definition \(\eta_1\) and \(\eta_2\) are independent Poisson point processes, and we can write
	\[
		\eta_1=\{\x_1,\dots,\x_{N_1}\}, \quad \text{ and }\quad \eta_2=\{\y_1,\dots,\y_{N_2}\},
	\]
	where \(N_1\) and \(N_2\) are two independent Poisson variables with parameters \(\lambda \omega_d\alpha\) resp.\ \(\lambda(3^d-2^d)\omega_d\alpha\). By the standard Chernoff bound for Poisson random variables, we have
	 \begin{equation}\label{eq:NumberOfPoints}
	 	\P^\lambda(N_1<c\alpha)+\P^\lambda(N_2<c\alpha)\leq e^{-c'\alpha},
	 \end{equation}
	 where \(c,c'>0\) are suitable constants depending on \(\lambda\). Since \(\lambda\) will play no more particular role in the remainder of the proof, let us for convenience assume that \(c=1\) to shorten notation. Let us further assume \(\alpha\in\N\) from now on. With this additional assumptions, \eqref{eq:NumberOfPoints} allows us to assume that both \(\eta_1\) and \(\eta_2\) contain  at least \(\alpha\)-many vertices.
	
	Before turning to calculating the probability of \(G'(\alpha)\), let us collect some properties of the involved vertex marks. To this end, let \(U=(U_1,\dots,U_\alpha)\) be a collection of \(\alpha\)-many independent random variables distributed uniformly on \((0,1)\). Let us denote their joint distribution by \(P\). Let us further define for all \(i=1,\dots,\lfloor\alpha^{1-\mu}\rfloor\) 
	\[
		N_U^\alpha(i):= \sum_{j=1}^\alpha \1\big\{U_j\leq\frac{i}{\lfloor\alpha^{1-\mu}\rfloor}\big\}.
	\]
	We say that \(U\) is \emph{\(\mu\)-regular} if
	\[
		N_U^\alpha(i)\geq \frac{i\alpha}{2\lfloor\alpha^{1-\mu}\rfloor}, \quad \forall \, i=1,\dots,\lfloor\alpha^{1-\mu}\rfloor.
	\]
	Since a simple calculation yields
	\[
		E [N_U^\alpha(i)] = \frac{i\alpha}{\lfloor\alpha^{1-\mu}\rfloor},
	\]
	we have by a Chernoff bound for independent Bernoulli random variables
	\[
		P\big(N_U^\alpha(i)<\tfrac{1}{2}E[N_U^\alpha(i)]\big)\leq e^{-c i \alpha^\mu}
	\]
	for some constant \(c>0\). Therefore, a union bound yields
	\begin{equation}\label{eq:muReg}
		\begin{aligned}
			P(U \text{ is not }\mu\text{-regular})\leq \alpha^{1-\mu}e^{-c\alpha^\mu}.
		\end{aligned}
	\end{equation}
	The idea of \(\mu\)-regularity is to gain bounds on the empirical distribution function of \(U\) which we denote by \(F\). Then, on the event of \(\mu\)-regularity, we have for \(t\in(0,1)\)
	\begin{equation}\label{eq:empiricalDistr}
		\begin{aligned}
			\alpha F(t) & = \sum_{i=1}^\alpha \1\{u_i\leq t\} \geq \sum_{i=1}^{\lfloor\alpha^{1-\mu}\rfloor}N_1^\alpha(i-1)\1\big\{i-1<\lfloor\alpha^{1-\mu}\rfloor t\leq i\big\} = N_1^\alpha(\lfloor t \lfloor\alpha^{1-\mu}\rfloor\rfloor) \\
			& \geq \frac{\alpha \lfloor t\lfloor\alpha^{1-\mu}\rfloor\rfloor}{2\lfloor\alpha^{1-\mu}\rfloor}\geq \frac{\alpha}{2}\Big(t-\frac{1}{\lfloor\alpha^{1-\mu}\rfloor}\Big)\geq \frac{\alpha}{3}(t-\alpha^{\mu-1}).		
		\end{aligned}
	\end{equation}
	
	With this at hand, we are ready to calculate the probability of \(G'(\alpha)\). 
	Let us denote the vertices \(\x_i\in\eta_1\) by \(\x_i=(x,u_i)\) and the vertices \(\y_j\in\eta_2\) by \(\y_j=(y_j,v_j)\). On the event \(N_1,N_2\geq\alpha\), we say that \(\eta_1\) is \emph{\(\mu\)-regular} if the collection of marks \((u_1,\dots,u_\alpha)\) is \(\mu\)-regular and we say the same on \(N_2\geq \alpha\) about \(\eta_2\) if \((v_1,\dots,v_\alpha)\) is \(\mu\)-regular. Let us denote by \(A\) the event that \(N_1,N_2\geq \alpha\) and that \(\eta_1,\eta_2\) are \(\mu\)-regular. Combining~\eqref{eq:NumberOfPoints} and~\eqref{eq:muReg}, we obtain for any \(\varepsilon>0\) and sufficiently large \(\alpha\)
	\[
		\P^\lambda(A^c)\leq 2e^{-c\alpha} + 2\P^\lambda(\eta_1 \text{ not }\mu\text{-regular}\mid N_1\geq \alpha) \leq C\alpha^{1-\mu}e^{-c\alpha^{\mu}}\leq \varepsilon.
	\]
	Therefore, for \(\varepsilon>0\) and large enough \(\alpha\)
	\[
		\P^\lambda(G'(\alpha))\geq (1-\varepsilon)\P^\lambda(G'(\alpha)\mid A) \geq (1-\varepsilon)\P^\lambda\big(\exists i,j\in\{1,\dots,\alpha\}:\x_i\sim\y_j \, \big| \, A\big).
	\]
	We finish the proof by showing that the probability on the right-hand side converges to one which is equivalent to showing that the probability of the complement converges to zero. Since \(|x_i-y_j|\leq 8\alpha^{1/d}\) for all \(\x_i\in\eta_1\), and \(\y_j\in\eta_2\) together with the fact that \(\varphi\) is non increasing in the third argument, we have
	\[
		1-\varphi(u_i,v_j,|x_i-y_j|^d)\leq e^{-\varphi(u_i,v_j,8^d \alpha)}.
	\]
	Therefore, writing \(F_1\) (resp.\ \(F_2\)) for the empirical distribution function of the marks \(u_1,\dots,u_\alpha\) (resp.\ \(v_1,\dots,v_\alpha\)), we infer 
	\begin{equation*}
		\begin{aligned}
			\P^{\lambda}\Big(\bigcap_{i,j=1}^\alpha\{\x_i\not\sim\y_j\}\cap A\Big)&\leq \E^\lambda\Big[\1_A \prod_{\substack{i,j=1}}^\alpha e^{-\varphi(u_i,v_j,8^d\alpha)}\Big] \\ 
			& =\E^\lambda\Big[\1_A \exp\Big(-\sum_{i,j=1}^\alpha \varphi(u_i,v_j,8^d\alpha)\Big)\Big] \\
			& = \E^\lambda\Big[\1_A \exp\Big(-\alpha^2\int\int\varphi(u,v,8^d\alpha) \, F_1(\d u) \, F_2(\d v)\Big)\Big] \\
			& \leq \P^\lambda(A)\exp\Big(-C\alpha^2 \int_{\alpha^{\mu-1}}^{1-\alpha^{{\mu-1}}} \int_{\alpha^{\mu-1}}^{1-\alpha^{\mu-1}}\varphi(u,v,8^d\alpha) \, \d u \, \d v\Big) \\
			&\leq \P^\lambda(A)\exp\big(-C \alpha^{2-\psi^+(\mu)-\varepsilon'}\big),
		\end{aligned}
	\end{equation*}
	for some \(\varepsilon'\) small and hence \(\alpha\) large enough and where we used~\eqref{eq:empiricalDistr} in the second to last step. This concludes the proof.
\end{proof}

\paragraph{Acknowledgement.} We gratefully received support by the Leibniz Association within the Leibniz Junior Research Group on \textit{Probabilistic Methods for Dynamic Communication Networks} as part of the Leibniz Competition (grant no.\ J105/2020).

\footnotesize{\printbibliography}
\end{document}